\documentclass[12pt,fleqn]{article}

\usepackage{amsmath,amssymb,amsthm}
\usepackage{geometry}
\usepackage{colortbl}

\newcommand{\bdry}[1]{\partial #1}
\newcommand{\CPS}[1]{$(CPS)_{#1}$}
\newcommand{\D}{{\cal D}}
\newcommand{\J}{{\cal J}}
\newcommand{\closure}[1]{\overline{#1}}
\newcommand{\comp}{\circ}
\newcommand{\dint}{\ds{\int}}
\newcommand{\ds}[1]{\displaystyle #1}
\newcommand{\eps}{\varepsilon}
\newcommand{\norm}[2][]{\left\|#2\right\|_{#1}}

\renewcommand{\o}{\text{o}}

\newcommand{\R}{\mathbb R}
\newcommand{\N}{\mathbb N}
\newcommand{\seq}[1]{\left(#1\right)}
\newcommand{\set}[1]{\left\{#1\right\}}
\newcommand{\wCPS}[1]{$(wCPS)_{#1}$}

\DeclareMathOperator*{\esssup}{ess\; sup}
\DeclareMathOperator{\divg}{div}

\newenvironment{enumroman}{\begin{enumerate}
\renewcommand{\theenumi}{$(\roman{enumi})$}
\renewcommand{\labelenumi}{$(\roman{enumi})$}}{\end{enumerate}}

\newenvironment{properties}[1]{\begin{enumerate}
\renewcommand{\theenumi}{$(#1_\arabic{enumi})$}
\renewcommand{\labelenumi}{$(#1_\arabic{enumi})$}}{\end{enumerate}}

\newtheorem{corollary}{Corollary}[section]
\newtheorem{lemma}[corollary]{Lemma}
\newtheorem{theorem}[corollary]{Theorem}

\theoremstyle{definition}
\newtheorem{definition}[corollary]{Definition}

\theoremstyle{remark}
\newtheorem{example}[corollary]{Example}
\newtheorem{remark}[corollary]{Remark}

\numberwithin{equation}{section}

\begin{document}

\title{\bf Existence results for a borderline case\\
of a class of $p$-Laplacian problems}

\author{Anna Maria Candela$^1$, Kanishka Perera$^2$ and Addolorata Salvatore$^3$\\
{\small $^{1,3}$Dipartimento di Matematica, Universit\`a degli Studi di Bari Aldo Moro} \\
{\small Via E. Orabona 4, 70125 Bari, Italy}\\
{\small $^2$Department of Mathematics and Systems Engineering, Florida Institute of Technology}\\
{\small Melbourne, FL 32901, USA}\\
{\small \it $^1$annamaria.candela@uniba.it, $^2$kperera@fit.edu, $^3$addolorata.salvatore@uniba.it}}
\date{}

\maketitle


\begin{abstract}
The aim of this paper is investigating the existence of at least one nontrivial  
bounded solution of the new asymptotically ``linear'' problem
\[ 
\left\{
\begin{array}{ll}
- \divg \left[\left(A_0(x) + A(x) |u|^{ps}\right) |\nabla u|^{p-2} \nabla u\right] 
+ s\ A(x) |u|^{ps-2} u\ |\nabla u|^p &\\[5pt]
\qquad\qquad\qquad =\ \mu |u|^{p (s + 1) -2} u + g(x,u) & \hbox{in $\Omega$,}\\[10pt]
u = 0 & \hbox{on $\bdry{\Omega}$,}
\end{array}
\right.
\]
where $\Omega$ is a bounded domain in $\R^N$, $N \ge 2$, $1 < p < N$, $s > 1/p$, 
both the coefficients $A_0(x)$ and $A(x)$ are in $L^\infty(\Omega)$ 
and far away from 0, $\mu \in \R$, 
and the ``perturbation'' term
$g(x,t)$ is a Carath\'{e}odory function on $\Omega \times \R$
which grows as $|t|^{r-1}$ with $1\le r < p (s + 1)$ and is such that
$g(x,t) \approx \nu |t|^{p-2} t$ as $t \to 0$.

By introducing suitable thresholds for the parameters $\nu$ and $\mu$,
which are related to the coefficients $A_0(x)$, respectively $A(x)$,
under suitable hypotheses on $g(x,t)$, the existence of a nontrivial weak solution
is proved if either $\nu$ is large enough with $\mu$ small enough
or $\nu$ is small enough with $\mu$ large enough.

Variational methods are used and in the first case 
a minimization argument applies while in the second case a suitable Mountain Pass 
Theorem is used.  
\end{abstract}




\section{Introduction}

In the last years, an increasing interest has been 
devoted to the study of (nontrivial bounded) solutions 
of the quasilinear elliptic $p$--Laplacian type problem
\begin{equation} \label{1.1}
\hspace{-2pt} 
\left\{\begin{array}{ll}
- \divg [(A_0(x) + A(x) |u|^{ps}) |\nabla u|^{p-2} \nabla u] 
+ s\ A(x) |u|^{ps-2} u\ |\nabla u|^p &\\[5pt]
 \qquad\qquad\qquad =\ \mu |u|^{q -2} u + g(x,u) & \hbox{in $\Omega$,}\\[10pt]
u = 0 &\hbox{on $\bdry{\Omega}$,}
\end{array}
\right.
\end{equation}
where $\Omega$ is a bounded domain in $\R^N$, $N \ge 2$,
$1 < p < N$, $s > 1/p$, $A_0(x)$ and $A(x)$ are given functions on $\Omega$, 
$1 \le q \le p^\ast (s + 1)$ with $p^\ast = Np/(N - p)$ critical Sobolev exponent, 
$\mu \in \R$, and $g: \Omega \times \R \to \R$ 
is a given nonlinear term which grows with a power strictly less than $q-1$. 

From a physical point of view, problem \eqref{1.1} is interesting for its applications.
In fact, if $\Omega = \R^N$ 
model equations of \eqref{1.1} appear in Mathematical Physics 
and describe several physical phenomena in the theory of superfluid film and in dissipative
quantum mechanics (for more details, see, e.g., \cite{LWW} and references therein)
and a possible approach to the study of solutions of equation \eqref{1.1}
on $\R^N$ requires an approximation scheme related to the existence of solutions  
in bounded domains (see, e.g., \cite{CSS2022}).  

If $s=0$ problem \eqref{1.1} has been widely investigated
by using variational methods
(see, e.g., \cite{PAO} and references therein), 
but if $s > 0$ and coefficient $A(x)$ is nontrivial,
the ``natural'' functional $\J$ related to problem \eqref{1.1}
is not $C^1$ in $W^{1,p}_0(\Omega)$ as it 
is G\^ateaux differentiable only along directions of 
$W^{1,p}_0(\Omega) \cap L^\infty(\Omega)$. 

In the past, such a problem has been overcome by introducing suitable definitions
of critical point (see, e.g., \cite{AB1,AB2,Ca,PeSq}). 

Here, following the ideas in \cite{CP2009}, we consider 
the functional $\J$ defined in 
\[
X = W^{1,p}_0(\Omega) \cap L^\infty(\Omega)
\]
so that, under suitable assumptions, it is $C^1$.

Then, a weak solution of this problem is a function $u$ belonging 
to the Banach space $X$ that satisfies
\[
\begin{split}
&\int_\Omega \left[\left(A_0(x) + A(x)\ |u|^{ps}\right) |\nabla u|^{p-2} \nabla u \cdot \nabla v 
+ s\ A(x)\ |u|^{ps-2} u\ v\ |\nabla u|^p \right] dx\\[5pt]
&\qquad \qquad = \mu \int_\Omega |u|^{q-2} u\ v\ dx + 
\int_\Omega g(x,u)\ v\ dx \qquad \forall v \in X.
\end{split}
\]

The purpose of this paper is to study the existence of nontrivial weak solutions
of problem \eqref{1.1} once $q\ge 1$ is given. 

The case $p (s + 1) < q < p^\ast (s + 1)$ was studied in Candela et al.\! \cite{MR4142333}
by using a Mountain Pass argument, on the contrary 
the case $1 \le q < p(s + 1)$ was investigated in \cite{CS2020}
by means of a Minimum Theorem.
So, here we focus on the borderline case $q = p (s + 1)$ and 
we note that, when $s > p/(N - p)$, it has to be
$p (s + 1) > p^\ast$ 
and hence problem \eqref{1.1} may also be supercritical.

If we replace $A_0(x) + A(x)\ |u|^{ps}$ 
with a generic ${\cal A}(x,u)$ which has not an explicit growth depending on $u$, 
the related equation becomes 
\[
- \divg [{\cal A}(x,u) |\nabla u|^{p-2} \nabla u] 
+ \frac{1}{p} {\cal A}_u(x,u)\ |\nabla u|^p = \mu |u|^{q -2} u + g(x,u)\quad \hbox{in $\Omega$}
\]
and the corresponding asymptotically $p$--linear case requires $q=p$.
In this setting, if symmetry occurs, the existence of multiple solutions 
have been proved since the seminal papers for $p=2$ and ${\cal A}(x,u) \equiv 1$
(see \cite{AZ,BBF}) till the most recent ones for $p>1$ and 
${\cal A}_u(x,u) \not\equiv 0$ in the non resonant case (see \cite{CPP2015} if $p>N$, 
\cite{CP2017} if $1< p\le N$).  
Here, as ${\cal A}(x,u) = A_0(x) + A(x)\ |u|^{ps}$, 
there is a new asymptotically ``linear'' case 
which depends not only from $p$ but also from $s$,
namely $q=p (s+1)$, and as a first step in investigating the properties 
of this new threshold, we look for the existence of at least one nontrivial
weak bounded solution when no symmetry occurs. 

More precisely, for $q = p (s + 1)$ with $s > 1/p$, 
we have that problem \eqref{1.1} reduces to
\begin{equation} \label{1.2}
\left\{
\begin{array}{ll}
- \divg \left[\left(A_0(x) + A(x) |u|^{ps}\right) |\nabla u|^{p-2} \nabla u\right] 
+ s\ A(x) |u|^{ps-2} u\ |\nabla u|^p &\\[5pt]
\qquad\qquad\qquad =\ \mu |u|^{p (s + 1) -2} u + g(x,u) & \hbox{in $\Omega$,}\\[10pt]
u = 0 & \hbox{on $\bdry{\Omega}$.}
\end{array}
\right.
\end{equation}
Assume that the coefficients $A_0$, $A: \Omega \to \R$
and the nonlinear term $g: \Omega \times \R \to \R$ are such that:
\begin{properties}{h}
\item \label{h1} 
$A_0\in L^\infty(\Omega)$, $A \in L^\infty(\Omega)$, and $\alpha_0 > 0$ exists so that
\[
A_0(x) \ge \alpha_0\quad \hbox{and}\quad 
A(x) \ge \alpha_0\qquad \hbox{for a.e.\! $x \in \Omega$;}
\]
\item \label{h2} $g(x,t)$ is a Carath\'{e}odory function in $\Omega \times \R$, 
i.e., $g(\cdot,t)$ is measurable 
for all $t \in \R$ and $g(x,\cdot)$ 
is continuous for a.e. $x \in \Omega$, and 
some constants $a_1$, $a_2 > 0$ and an exponent $1 \le r < p(s+1)$ exist such that
\[
|g(x,t)| \le a_1 + a_2\ |t|^{r-1} \quad\hbox{for a.e.\! $x \in \Omega$, all $t \in \R$;}
\]
\item \label{h3} a parameter $\nu \in \R$ exists such that
\[
g(x,t) = \nu |t|^{p-2}\ t + \o(|t|^{p-1})\quad \hbox{as $t \to 0$, uniformly a.e.\! in $\Omega$.}
\]
\end{properties}

Firstly, we consider the asymptotic eigenvalue problems
\[
\left\{
\begin{array}{ll}
- \divg \left(A_0(x)\ |\nabla u|^{p-2}\ \nabla u\right) = \nu |u|^{p-2} u & \text{in } \Omega\\[10pt]
u = 0 & \text{on } \bdry{\Omega}
\end{array}\right.
\]
and
\[
\left\{
\begin{array}{ll}
- \divg \left(A(x)\ |u|^{ps}\ |\nabla u|^{p-2}\ \nabla u\right) 
+ s A(x)\ |u|^{ps-2} u\ |\nabla u|^p = \mu |u|^{p(s+1)-2} u & \text{in } \Omega\\[10pt]
u = 0 & \text{on } \bdry{\Omega}
\end{array}\right.
\]
which will play a role in our results. 
Thus, we introduce the Rayleigh quotients on $X$ associated with these problems, which are
$R$, $S : X \setminus \set{0} \to \R$ defined as
\[
R(u) = \frac{\dint_\Omega A_0(x) |\nabla u|^p dx}{\dint_\Omega |u|^p dx}, 
\qquad S(u) = \frac{(s + 1) \dint_\Omega A(x)\ |u|^{ps}\ |\nabla u|^p dx}{\dint_\Omega |u|^{p (s+1)} dx},
\]
respectively. 
Set
\begin{equation} \label{1.3}
\nu_1 = \inf_{u \in X \setminus \set{0}} R(u), \qquad \mu_1 = \inf_{u \in X \setminus \set{0}} S(u).
\end{equation}
For any $u \in X$, we have that 
\begin{equation}
\label{equal}
|\nabla (|u|^s u)|^p\ =\ (s + 1)^p\ |u|^{ps}\ |\nabla u|^p\quad
\hbox{a.e.\! in $\Omega$}
\end{equation}
and, hence, $|u|^s u \in W^{1,p}_0(\Omega)$. 
From \eqref{equal} and \ref{h1} it follows that
\[
\nu_1\ \ge\ \alpha_0 \lambda_1, \qquad 
\mu_1\ \ge\ \frac{\alpha_0}{(s+1)^{p-1}}\ \lambda_1
\]
with 
\begin{equation}
\label{first}
\lambda_1\ =\ \inf_{u \in W^{1,p}_0(\Omega) \setminus \set{0}}\
\frac{\dint_\Omega |\nabla u|^p dx}{\dint_\Omega |u|^p dx}\ >\ 0
\end{equation}
first eigenvalue of $- \Delta_p$ in $W^{1,p}_0(\Omega)$.
Thus, $\nu_1$, $\mu_1 > 0$. 
\smallskip

Now, we can state our first existence result.

\begin{theorem} \label{Theorem 1.1}
Assume \ref{h1}--\ref{h3}. 
If $\nu > \nu_1$ and $\mu < \mu_1$, 
then problem \eqref{1.2} has a nontrivial weak bounded solution.
\end{theorem}

When $\mu \ge \mu_1$, we need an additional assumption on the nonlinearity $g(x,t)$. 
To this aim, let
\[
G(x,t) = \int_0^t g(x,\tau)\ d\tau
\]
be the primitive of $g(x,t)$ with $G(x,0)=0$ for a.e. $x \in \Omega$, 
and note that \ref{h2} gives
\begin{equation} \label{1.4}
|G(x,t)| \le a_3 + a_4\, |t|^r \quad \text{for a.e. } x \in \Omega, \text{ all } t \in \R
\end{equation}
for some constants $a_3, a_4 > 0$. Set
\[
H(x,t) = \frac{1}{s}\ \big[p (s + 1)\ G(x,t) - g(x,t)\ t \big], 
\quad (x,t) \in \Omega \times \R.
\]
We assume that
\begin{properties}{h}
\setcounter{enumi}{3}
\item \label{h4} some constants $0 \le \vartheta < \nu_1$ and $a_5 > 0$
exist so that
\[
H(x,t) \le \vartheta\ |t|^p + a_5\quad \hbox{for a.e.\! $x \in \Omega$, all $t \in \R$.}
\]
\end{properties}

\begin{remark}
If $r < p$ condition \ref{h4} follows from \ref{h2} and \eqref{1.4} 
just taking any constant $0 < \vartheta < \nu_1$. 
\end{remark}

Our second result is the following theorem.

\begin{theorem} \label{Theorem 1.2}
Assume \ref{h1}--\ref{h4}. 
If $\nu < \nu_1$ and $\mu > \mu_1$, 
then problem \eqref{1.2} has a nontrivial weak bounded solution.
\end{theorem}

\begin{example}\label{ex1}
The function
\[
g(x,t) = \nu\ |t|^{p-2} t + b(x)\ |t|^{r-2} t
\]
satisfies \ref{h2}--\ref{h3} if $b \in L^\infty(\Omega)$ and $p < r < p (s + 1)$.\\
Moreover, if $\nu < \nu_1$ and coefficient $b(x)$ is nonpositive, 
condition \ref{h4} holds.
\end{example}

We note that if assumptions \ref{h1}--\ref{h3} are satisfied 
but both $\nu \le \nu_1$ and $\mu \le \mu_1$ hold,
then some models of problem \eqref{1.2} admits only the trivial solution
(see Remark \ref{trivial}).

Proofs of Theorems \ref{Theorem 1.1} and \ref{Theorem 1.2} will be given in Section \ref{Section 3}
after some preliminaries introduced in the next sections.


\section{Abstract settings} \label{abstract}

In this section we recall from \cite{MR4142333} 
the abstract tools that we need to prove our theorems. 

Let $(X,\norm[X]{\cdot})$ be a Banach space with dual $(X',\norm[X']{\cdot})$. 
Assume that $X$ is continuously embedded in a Banach space $(W,\norm[W]{\cdot})$. 
Let $J : \D \subset W \to \R$ be a functional such that $X \subset \D$ and $J \in C^1(X,\R)$.

For $\beta \in \R$, let
\[
K_J^\beta = \set{u \in X : J(u) = \beta,\, dJ(u) = 0}
\]
be the set of critical points of $J$ in $X$ at the level $\beta$ and let
\[
J^\beta = \set{u \in X : J(u) \le \beta}
\]
be the sublevel set of $J$ with respect to $\beta$.

\begin{definition}
We say that $J$ satisfies the weak Cerami-Palais-Smale condition in $X$
at the level $\beta \in \R$, or the \wCPS{\beta} condition for short, 
if for every $(CPS)_\beta$--sequence $\seq{u_n}_n \subset X$, i.e. a sequence
such that
\[
J(u_n) \to \beta, \qquad \norm[X']{dJ(u_n)} (1 + \norm[X]{u_n}) \to 0,
\]
there exists $u \in K_J^\beta$ such that 
$\norm[W]{u_n - u} \to 0$ for a renamed subsequence.
\end{definition}

We have the following deformation lemma (see \cite[Lemma 2.2]{MR4142333}).

\begin{lemma} \label{Lemma 2.1}
If $\beta \in \R$ is such that $J$ satisfies the \wCPS{\beta} 
condition and $K_J^\beta = \emptyset$, 
then for any $\bar{\eps} > 0$, there exist $0 < \eps < \bar{\eps}$ 
and a homeomorphism $\psi : X \to X$ such that
\begin{enumroman}
\item $\psi(J^{\beta + \eps}) \subset J^{\beta - \eps}$,
\item $\psi(u) = u$ if $J(u) \le \beta - \bar{\eps}$ or $J(u) \ge \beta + \bar{\eps}$.
\end{enumroman}
\end{lemma}

In particular, this lemma gives a minimizer when $J$ is bounded from below.

\begin{theorem} \label{Theorem 2.3}
If
\begin{equation} \label{2.1}
\beta := \inf_{u \in X}\ J(u) > - \infty
\end{equation}
and $J$ satisfies the \wCPS{\beta} condition, then $K_J^\beta \ne \emptyset$.
\end{theorem}

\begin{proof}
Suppose $K_J^\beta = \emptyset$. Then Lemma \ref{Lemma 2.1} 
gives a constant $\eps > 0$ and a homeomorphism $\psi : X \to X$ 
such that $\psi(J^{\beta + \eps}) \subset J^{\beta - \eps}$. 
Since $J^{\beta + \eps} \ne \emptyset$ by \eqref{2.1}, then $J^{\beta - \eps} \ne \emptyset$, 
contradicting \eqref{2.1}.
\end{proof}

Lemma \ref{Lemma 2.1} also implies a generalized version of the Ambrosetti-Rabinowitz 
Mountain Pass Theorem (see \cite{AR}).

\begin{theorem} \label{Theorem 2.4}
Assume that there exist a neighborhood $U$ of the origin in $X$ 
and $u_1 \in X \setminus \closure{U}$ such that
\[
\max \set{J(0), J(u_1)}\ <\ \inf_{u \in \bdry{U}}\ J(u).
\]
Let
\[
\Gamma\ =\ \set{\gamma \in C([0,1],X) :\ \gamma(0) = 0,\, \gamma(1) = u_1}
\]
be the class of paths in $X$ joining $0$ and $u_1$, and set
\begin{equation} \label{2.2}
\beta := \inf_{\gamma \in \Gamma}\ \max_{u \in \gamma([0,1])}\ J(u)\ \ge\ \inf_{u \in \bdry{U}} J(u).
\end{equation}
If $J$ satisfies the \wCPS{\beta} condition, then $K_J^\beta \ne \emptyset$.
\end{theorem}

\begin{proof}
Suppose $K_J^\beta = \emptyset$ and
let $0 < \bar{\eps} \le \beta - \max \set{J(0), J(u_1)}$. 
Then, a value $0 < \eps < \bar{\eps}$
and a homeomorphism $\psi : X \to X$ exist as in Lemma \ref{Lemma 2.1}. 
On the other hand,
by \eqref{2.2}, $\exists\, \gamma \in \Gamma$ such that $\max J(\gamma([0,1])) \le \beta + \eps$. \\
Now, let $\widetilde{\gamma} = \psi \comp \gamma$. 
Since $J(0), J(u_1) \le \beta - \bar{\eps}$, we have that
$\psi(0) = 0$ and $\psi(u_1) = u_1$, 
so $\widetilde{\gamma} \in \Gamma$. 
However, $\max J(\widetilde{\gamma}([0,1])) \le \beta - \eps$ since 
$\psi(J^{\beta + \eps}) \subset J^{\beta - \eps}$, contradicting \eqref{2.2}.
\end{proof}


\section{Preliminaries} 

From now on, let $\Omega \subset \R^N$ be an open bounded domain, $N\ge 2$,
and consider $1 < p < N$.
We denote by:
\begin{itemize}
\item $L^q(\Omega)$ the Lebesgue space with
norm $|u|_q = \left(\int_\Omega|u|^q dx\right)^{1/q}$ if $1 \le q < +\infty$;
\item $L^\infty(\Omega)$ the space of Lebesgue--measurable 
and essentially bounded functions $u :\Omega \to \R$ with norm
\[
|u|_{\infty} = \esssup_{\Omega} |u|;
\]
\item $W^{1,p}_0(\Omega)$ the classical Sobolev space with
norm $\|u\|_{W} = |\nabla u|_p$;
\item $|C|$ the usual Lebesgue measure of a measurable set $C$ in $\R^N$.
\end{itemize}

In order to investigate the existence of weak solutions  
of the nonlinear problem \eqref{1.2}, the notation introduced for the abstract 
setting at the beginning of Section \ref{abstract}
is referred to our problem with $W = W^{1,p}_0(\Omega)$ and the
Banach space $(X,\|\cdot\|_X)$ is defined as in Introduction,
i.e., 
\[
X = W^{1,p}_0(\Omega) \cap L^\infty(\Omega),\qquad
\|u\|_X = \|u\|_W + |u|_\infty.
\]
Here and in the following, $|\cdot|$ denotes 
the standard norm on any Euclidean space as the dimension
of the considered vector is clear and no ambiguity arises;
moreover, $(\eps_n)_n$ represents any infinitesimal sequence 
depending on a given sequence $(u_n)_n$.

From the definition of $X$, we have that 
$X \hookrightarrow W^{1,p}_0(\Omega)$ and $X \hookrightarrow L^\infty(\Omega)$
with continuous embeddings. 

If hypotheses \ref{h1} - \ref{h2} holds, then for all $s > \frac1p$
and $\mu \in \R$
the functional $\J : X \to \R$ defined as
\[
\begin{split}
\J(u)\ =\ &\frac{1}{p} \int_\Omega \left(A_0(x) + A(x)\ |u|^{ps}\right) |\nabla u|^p\ dx
 - \frac{\mu}{p (s + 1)}\ \int_\Omega |u|^{p (s+1)}\ dx\\[5pt]
& -\ \int_\Omega G(x,u)\ dx
\end{split}
\]
is well defined and of class $C^1$ in $X$ (see \cite[Proposition 3.2]{MR4142333})

Thus, weak solutions of problem \eqref{1.2} coincide 
with critical points of $\J$ in $X$, where
\[
\begin{split}
\langle d\J(u),v\rangle\ =\ &\int_\Omega \left[\left(A_0(x) + A(x)\ |u|^{ps}\right) 
|\nabla u|^{p-2} \nabla u \cdot \nabla v + s\ A(x)\ |u|^{ps-2} u\ v |\nabla u|^p \right] dx\\[5pt]
&-\ \mu \int_\Omega |u|^{p(s+1)-2} u\ v\ dx - 
\int_\Omega g(x,u)\ v\ dx \qquad \forall\, u,\ v \in X.
\end{split}
\]

As useful in the next proofs, 
let us point out the following technical lemmas
(for the first one, see \cite[Lemma 3.8]{MR4142333}).

\begin{lemma}\label{rellich}
Taking $1 < p < N$ and $s > \frac{1}{p}$, 
let $(u_n)_n \subset X$ be a sequence such that
\begin{equation}\label{c1}
\left( \int_\Omega (1+|u_n|^{p s})\ |\nabla u_n|^p dx\right)_n\quad \hbox{is bounded.} 
\end{equation}
Then, $u \in W^{1,p}_0(\Omega)$ exists such that 
$|u|^s u \in W^{1,p}_0(\Omega)$, too, and, up to subsequences, as $n\to+\infty$ we have
\begin{eqnarray}
&&u_n \rightharpoonup u\ \hbox{weakly in $W^{1,p}_0(\Omega)$,}
\label{c2}\\
&&|u_n|^s u_n \rightharpoonup |u|^s u\ \hbox{weakly in $W^{1,p}_0(\Omega)$,}
\label{c7}\\
&&u_n \to u\ \hbox{a.e. in $\Omega$,}
\label{c3}\\
&&u_n \to u\ \hbox{strongly in $L^m(\Omega)$ for each $m \in [1,p^*(s+1)[$.}
\label{c4}
\end{eqnarray}
\end{lemma}

\begin{lemma}\label{upperG}
Assume that \ref{h2} holds. 
Then, fixing any $\mu$, $\widetilde{\mu}\in \R$
with $\widetilde{\mu} > \mu$, a constant $a_{\mu, \widetilde{\mu}}> 0$
(which depends on $\mu$ and $\widetilde{\mu}$) exists so that 
\[
\int_\Omega G(x,u) dx\ \le\ \frac{\widetilde{\mu} - \mu}{p (s + 1)}\ 
\int_\Omega |u|^{p (s+1)}\ dx + a_{\mu, \widetilde{\mu}}
\qquad \hbox{for all $u \in X$.}
\]
\end{lemma}

\begin{proof}
Fix any $u \in X$. Since $1 \le r < p (s + 1)$, from \eqref{1.4}, 
H\"{o}lder's inequality, and the Young's inequality 
it follows that
\[
\begin{split}
\int_\Omega G(x,u) dx\ &\le\ a_3 |\Omega| + a_4 \int_\Omega |u|^{r}\ dx\
\le\ a_3 |\Omega| + a_6 \left(\int_\Omega |u|^{p (s+1)}\ dx\right)^{\frac{r}{p (s+1)}}\\
&\le\ \frac{\widetilde{\mu} - \mu}{p (s + 1)}\ 
\int_\Omega |u|^{p (s+1)}\ dx + a_{\mu, \widetilde{\mu}}
\end{split}
\]
for a suitable constant $a_6> 0$, coming from H\"{o}lder's inequality,
which is independent of $\mu$ and $\widetilde{\mu}$,
and a constant $a_{\mu, \widetilde{\mu}}> 0$,
coming from Young's inequality, which depends on $\frac{\widetilde{\mu} - \mu}{p (s + 1)} > 0$.
\end{proof}

Now, we are ready to verify that $\J$ satisfies the \wCPS{\beta} condition 
in $X$ for any $\beta \in\R$.

\begin{lemma} \label{Lemma 3.1}
Assume that \ref{h1} and \ref{h2} hold. Then, $\J$ satisfies 
the \wCPS{\beta} condition in $X$ for all $\beta \in \R$ 
in each of the following cases:
\begin{enumroman}
\item \label{Lemma 3.1.i} $\mu < \mu_1$,
\item \label{Lemma 3.1.ii} $\mu \ge \mu_1$ and \ref{h4} holds.
\end{enumroman}
\end{lemma}

\begin{proof}
Let $\beta \in \R$ and let $\seq{u_n}_n \subset X$ be a \CPS{\beta} sequence. 
Then, we have that
\begin{equation}
\begin{split}
 \beta + \eps_n \ =\ \J(u_n) \ =\ 
&\frac{1}{p} \int_\Omega \left(A_0(x) + A(x)\ |u_n|^{ps}\right) |\nabla u_n|^p\ dx \\
& - \frac{\mu}{p (s + 1)} \int_\Omega |u_n|^{p (s+1)}\ dx \ -\ \int_\Omega G(x,u_n)\ dx 
\end{split}
\label{3.1}
\end{equation}
and
\begin{equation}
\begin{split}
\eps_n\ =\ \langle d\J(u_n),u_n\rangle\
= &\int_\Omega \left(A_0(x) + (s + 1)\ A(x)\ |u_n|^{ps}\right) |\nabla u_n|^p\ dx\\
& - \mu \int_\Omega |u_n|^{p (s+1)}\ dx\ -\ \int_\Omega g(x,u_n)\ u_n\ dx . 
\end{split}
\label{3.2}
\end{equation}
The first step of the proof is showing that \eqref{c1} holds, 
so we can apply Lemma \ref{rellich}. To this aim, we will consider the two different 
conditions.

\ref{Lemma 3.1.i} Fix $\widetilde{\mu} \in\ ]\mu,\mu_1[$.
Combining Lemma \ref{upperG} with \eqref{3.1} and \eqref{1.3} we have that
\[
\begin{split}
 \beta + \eps_n \ \ge\ & \frac{1}{p} \int_\Omega \left(A_0(x) + A(x)\ |u_n|^{ps}\right) |\nabla u_n|^p\ dx 
- a_{\mu, \widetilde{\mu}} 
- \frac{\widetilde{\mu}}{p (s + 1)} \int_\Omega |u_n|^{p (s+1)}\ dx \\
\ge\ & \frac{1}{p} \int_\Omega A_0(x)\ |\nabla u_n|^p\ dx 
+\ \frac{1}{p} \left(1 - \frac{\widetilde{\mu}}{\mu_1}\right)\ \int_\Omega A(x)\ |u_n|^{ps} |\nabla u_n|^p\ dx 
- a_{\mu, \widetilde{\mu}}. 
\end{split}
\]
Thus, a constant $a_7= a_7(\mu, \widetilde{\mu}) > 0$ exists so that
\[
\int_\Omega A_0(x)\ |\nabla u_n|^p dx + 
\left(1 - \frac{\widetilde{\mu}}{\mu_1}\right) \int_\Omega A(x)\ |u_n|^{ps}\ |\nabla u_n|^p\ dx\ \le \ a_7
\]
for all $n \in \N$, which together with \ref{h1} gives the desired conclusion \eqref{c1}.

\ref{Lemma 3.1.ii} First we show that 
\begin{equation}
\label{bdd0}
\left(\int_\Omega |\nabla u_n|^p\ dx\right)_n\qquad
\hbox{is bounded,}
\end{equation}
or better, by \ref{h1}, it suffices to show that 
\begin{equation}
\label{bdd1}
\left(\int_\Omega A_0(x)\ |\nabla u_n|^p\ dx\right)_n\qquad
\hbox{is bounded.}
\end{equation}
Combining \eqref{3.1} and \eqref{3.2} as $\langle d\J(u_n),u_n\rangle - p(s+1) \J(u_n)$ 
gives
\[
\eps_n - p(s+1) \beta\ =\ -\ s\ \int_\Omega A_0(x)\ |\nabla u_n|^p\ dx 
+ s\ \int_\Omega H(x,u_n)\ dx;
\] 
hence, from \ref{h4} and \eqref{1.3} it follows that
\begin{equation} \label{3.4}
\begin{split}
\int_\Omega A_0(x)\ |\nabla u_n|^p\ dx \ &\le\ \int_\Omega H(x,u_n)\ dx + a_8\
\le \
\vartheta \int_\Omega |u_n|^p\ dx + a_9 \\
&\le\ 
\frac{\vartheta}{\nu_1} \int_\Omega A_0(x)\ |\nabla u_n|^p\ dx + a_9
\end{split}
\end{equation}
for some constants $a_8$, $a_9 > 0$. Since $\vartheta < \nu_1$, \eqref{3.4} implies \eqref{bdd1}.\\
Now, we show that 
\[
\left(\int_\Omega |u_n|^{ps}\ |\nabla u_n|^p\ dx\right)_n\qquad
\hbox{is also bounded.}
\]
By \ref{h1} again, it suffices to show that
$\ (\rho_n)_n\ $ is bounded, with 
\[
\rho_n := \left(\int_\Omega A(x)\ |u_n|^{ps}\ |\nabla u_n|^p\ dx\right)^{\frac{1}{p(s+1)}}.
\]
Arguing by contradiction, suppose 
\begin{equation} \label{unbdd2}
\rho_n\ \to\ + \infty
\end{equation}
for a renamed subsequence, and set 
\[
\widetilde{u}_n = \frac{u_n}{\rho_n}.
\]
Since \eqref{bdd0} holds, from \eqref{unbdd2}
we have that 
\begin{equation} \label{unbdd2-1}
 \widetilde{u}_n \to 0\qquad \hbox{strongly in $W^{1,p}_0(\Omega)$.}
\end{equation}
On the other hand, since
\begin{equation} \label{3.5}
\int_\Omega A(x)\ |\widetilde{u}_n|^{ps}\ |\nabla \widetilde{u}_n|^p\ dx = 1 \qquad \hbox{for all $n\in \N$,}
\end{equation}
from \ref{h1} it follows that 
\[
\left(\int_\Omega |\widetilde{u}_n|^{ps}\ |\nabla \widetilde{u}_n|^p\ dx\right)_n\qquad
\hbox{is bounded.}
\]
Hence, sequence $(\widetilde{u}_n)_n$ satisfies the boundedness condition
\eqref{c1} and from Lemma \ref{rellich} a renamed subsequence of $\seq{\widetilde{u}_n}_n$ 
converges to some $\widetilde{u}$ weakly in $W^{1,p}_0(\Omega)$ and strongly 
in $L^m(\Omega)$ for all $m \in [1,p^\ast (s + 1)[$. 
From \eqref{unbdd2-1}, it has to be $\widetilde{u} = 0$. 
Thus, being $1 \le r < p (s + 1) < p^\ast (s + 1)$, we have that
\[
\frac{1}{\rho_n^{p (s+1)}} \ \int_\Omega |u_n|^{p(s+1)}\ dx
\ =\ \int_\Omega |\widetilde{u}_n|^{p(s+1)}\ dx\ =\ \eps_n,
\] 
while from \eqref{1.4} and \eqref{unbdd2} it follows that
\[
\frac{1}{\rho_n^{p (s+1)}} \left|\int_\Omega G(x,u_n)\ dx\right|\
\le\
\frac{a_3\ |\Omega|}{\rho_n^{p (s+1)}} + \frac{a_4}{\rho_n^{p (s+1) - r}}\
 \int_\Omega |\widetilde{u}_n|^r\ dx\ =\ \eps_n,
\] 
and from \eqref{bdd1} and \eqref{unbdd2} we obtain
\[
\frac{1}{\rho_n^{p (s+1)}} \int_\Omega A_0(x)\ |\nabla u_n|^p\ dx\ =\ \eps_n.
\] 
Hence, summing up all the previuos information, 
from dividing \eqref{3.1} by $\rho_n^{p (s+1)}$ and from \eqref{unbdd2}
it follows that 
\[
\int_\Omega A(x)\ |\widetilde{u}_n|^{ps}\ |\nabla \widetilde{u}_n|^p\ dx \to 0
\]
contradicting \eqref{3.5}.\\
So, also in this case the boundedness condition \eqref{c1} holds.

Now, in both cases \ref{Lemma 3.1.i} and \ref{Lemma 3.1.ii}, boundedness condition \eqref{c1} 
allows us to apply Lemma \ref{rellich} and a function $u \in W^{1,p}_0(\Omega)$ 
exists such that $|u|^s u \in W^{1,p}_0(\Omega)$
and \eqref{c2}--\eqref{c4} hold, up to subsequences.\\
Then, by the arguments in the proof of \cite[Proposition 3.10]{MR4142333}
it follows $u \in L^\infty(\Omega)$ and also $\|u_n - u\|_W \to 0$
and $\J(u) = \beta$, $d\J(u) = 0$.
\end{proof}


\section{Proofs of the main theorems} \label{Section 3}

Now, we are ready to prove Theorem \ref{Theorem 1.1}.

\begin{proof}[Proof of Theorem \ref{Theorem 1.1}]
As in the proof of Lemma \ref{Lemma 3.1} \ref{Lemma 3.1.i}, 
fix $\widetilde{\mu} \in\ ]\mu,\mu_1[$ and note that,
by combining Lemma \ref{upperG} with \eqref{1.3}, for any $u \in X$
we have that
\[
\J(u) \ge \frac{1}{p} \int_\Omega A_0(x)\ |\nabla u|^p\ dx\ 
+\ \frac{1}{p}\ \left(1 - \frac{\widetilde{\mu}}{\mu_1}\right)\ 
\int_\Omega A(x)\ |u|^{ps}\ |\nabla u|^p\ dx - a_{\mu,\widetilde{\mu}}
\]
for a constant $a_{\mu,\widetilde{\mu}} > 0$. So, \ref{h1} implies that
\[
\beta := \inf_{u \in X} \J(u) > - \infty.
\]
Since $\J$ satisfies the \wCPS{\beta} condition by Lemma \ref{Lemma 3.1} \ref{Lemma 3.1.i}, 
then Theorem \ref{Theorem 2.3} applies and $K_\J^\beta \ne \emptyset$.\\
We will complete the proof by showing that $\beta < 0$ and, hence, $0 \notin K_\J^\beta$.\\
Fix $\widetilde{\nu} \in\ ]\nu_1,\nu[$. 
We note that \ref{h3} implies
\[
\frac{G(x,t) - \frac{\nu}{p} |t|^p}{\frac{1}{p} |t|^p}\ \to \ 0 \quad \hbox{as $ t \to 0$,
uniformly a.e. in $\Omega$,}
\]
then, $\delta = \delta(\widetilde{\nu}) > 0$ exists such that
\begin{equation} \label{3.6}
G(x,t) \ge \frac{\widetilde{\nu}}{p}\ |t|^p\qquad
\hbox{for $|t| < \delta$ and a.e. $x \in \Omega$.}
\end{equation}
Now, let us consider $\varphi_1 \in W^{1,p}_0(\Omega)$
eigenfunction of $- \Delta_p$ in $W^{1,p}_0(\Omega)$ 
which achieves the eigenvalue $\lambda_1$ defined in \eqref{first}.
It is known that it is the unique function such that
\[
\varphi_1 > 0,\quad \int_{\Omega} |\varphi_1|^p dx = 1\quad\hbox{and}\quad  
\int_{\Omega} |\nabla \varphi_1|^p dx = \lambda_1
\]
(see, e.g., \cite{Lin}). Furthermore, it is also
$\varphi_1 \in L^\infty(\Omega)$, hence $\varphi_1 \in X$,
and from \ref{h1} it results
\[
\int_\Omega A_0(x)\ |\nabla \varphi_1|^p\ dx \ge \alpha_0 \lambda_1 > 0.
\]
Taking $\sigma > 0$ so that $0 < \sigma\ |\varphi_1|_\infty < \delta$,
from \eqref{3.6} and \eqref{1.3} it follows 
\[
\begin{split}
\J(\sigma \varphi_1)\ \le\ 
&\frac{\sigma^p}{p} \int_\Omega A_0(x)\ |\nabla \varphi_1|^p\ dx\ 
-\ \frac{\widetilde{\nu}}{p}\ \sigma^p\\
& + \frac{\sigma^{p (s+1)}}{p}
\ \left(\int_\Omega A(x)\ |\varphi_1|^{ps}\ |\nabla \varphi_1|^p\ dx
-\ \frac{\mu}{s + 1} \int_\Omega |\varphi_1|^{p (s+1)}\ dx\right) \\
\le\ &-\ \frac{\sigma^p}{p} \left(\frac{\widetilde{\nu}}{\nu_1} - 1\right) 
\ \int_\Omega A_0(x)\ |\nabla \varphi_1|^p\ dx \\
& + \frac{\sigma^{p (s+1)}}{p}
\ \left(\int_\Omega A(x)\ |\varphi_1|^{ps}\ |\nabla \varphi_1|^p\ dx
-\ \frac{\mu}{s + 1} \int_\Omega |\varphi_1|^{p (s+1)}\ dx\right);
\end{split}
\]
so, $\J(\sigma \varphi_1) < 0$ for all sufficiently small $\sigma > 0$. 
Hence, $\beta < 0$.
\end{proof}

Next, we prove a technical lemma useful for verifying the mountain pass geometry
we need for the proof of Theorem \ref{Theorem 1.2}. 

Let
\[
L: u \in X \  \mapsto\ L(u) = \int_\Omega \left(A_0(x) + A(x)\ |u|^{ps}\right) |\nabla u|^p\ dx \in \R.
\]
From \ref{h1}, the map $L : X \to \R$ is continuous with 
\[
L(u) \ge 0\quad \hbox{for all $u \in X$, and}
\quad L(u) =0\ \iff\ u = 0.
\]  
Then, for each $\rho > 0$ the set
\[
U_\rho = \set{u \in X :\ L(u) < \rho}
\]
is a neighborhood of the origin in $X$; 
moreover, for any $u_0 \in X \setminus \set{0}$ 
a sufficiently large $\sigma_*(u_0,\rho) > 0$ exists 
so that 
\[
L(\sigma u_0) > \rho\quad \hbox{for all}\; \sigma > \sigma_*(u_0,\rho),
\]
and, hence, 
\begin{equation} \label{mp0}
\sigma u_0\ \in\ X \setminus \closure{U}_\rho
\quad \hbox{for all}\; \sigma > \sigma_*(u_0,\rho).
\end{equation}

\begin{lemma} \label{Lemma 3.2}
Let
\[
I(u)\ =\ \frac{1}{p} \int_\Omega \left(A_0(x) + A(x)\ |u|^{ps}\right) |\nabla u|^p\ dx 
- \int_\Omega F(x,u)\ dx, \quad u \in X,
\]
where $F: \Omega \times \R \to \R$ 
is a Carath\'{e}odory function satisfying the growth condition
\begin{equation} \label{3.7}
|F(x,t)|\ \le\ a_{10} + a_{11}\ |t|^{p^\ast(s+1)} \quad 
\text{for a.e. } x \in \Omega, \text{ all } t \in \R
\end{equation}
for some constants $a_{10}, a_{11} > 0$. 
If
\begin{equation} \label{3.8}
\nu\ :=\ \limsup_{t \to 0}\ \frac{p\ F(x,t)}{|t|^p} < \nu_1 \quad \text{uniformly a.e.\! in } \Omega,
\end{equation}
then a radius $\rho_* > 0$ exists so that 
\[
\inf_{u \in \bdry{U_\rho}} I(u) > 0\qquad \hbox{for all $0 < \rho < \rho_*$.}
\]
\end{lemma}

\begin{proof}
Fix $\widetilde{\nu} \in \ ]\nu,\nu_1[$. 
From \eqref{3.7}, \eqref{3.8} and direct computations it follows that
\[
F(x,t)\ \le\ \frac{\widetilde{\nu}}{p}\ |t|^p + a_{12}\ |t|^{p^\ast(s+1)} 
\quad \text{for a.e. } x \in \Omega, \text{ all } t \in \R
\]
for a suitable constant $a_{12} > 0$; so,
\begin{equation} \label{3.9}
I(u)\ \ge\ \frac{1}{p} \int_\Omega \left(A_0(x) + A(x)\ |u|^{ps}\right) |\nabla u|^p\ dx
\ -\ \frac{\widetilde{\nu}}{p}\ \int_\Omega |u|^p\ dx - a_{12} \int_\Omega |u|^{p^\ast(s+1)}\ dx.
\end{equation}
Now, the Sobolev Embedding Theorem, \eqref{equal} and \ref{h1} imply
\begin{equation} \label{less}
\begin{split}
\int_\Omega |u|^{p^\ast(s+1)}\ dx &=\ 
\int_\Omega ||u|^s u|^{p^\ast}\ dx \ 
\le\ a_{13} \left(\int_\Omega |\nabla (|u|^s u)|^p\ dx\right)^{p^\ast/p}\\
&\le\ a_{14} \left(\int_\Omega A(x)\ |u|^{ps}\ |\nabla u|^p\ dx\right)^{p^\ast/p}\\
&\le\ a_{14}\ L(u)^{\frac{p}{N-p}} \int_\Omega A(x)\ |u|^{ps}\ |\nabla u|^p\ dx
\end{split}
\end{equation}
for some constants $a_{13}, a_{14} > 0$. Hence,   
by using \eqref{1.3} and \eqref{less} in \eqref{3.9} 
we obtain that
\[
I(u)\ \ge\ \frac{1}{p} \left(1 - \frac{\widetilde{\nu}}{\nu_1}\right) \int_\Omega A_0(x)\ |\nabla u|^p\ dx
+ \frac{1}{p}\ \Big(1 - a_{15}\ L(u)^{\frac{p}{N-p}}\Big) \int_\Omega A(x)\ |u|^{ps}\ |\nabla u|^p\ dx
\]
for a suitable constant $a_{15} > 0$.
Thus, taking $\rho_* > 0$ so that $a_{15}\ \rho_*^{\frac{p}{N-p}} < \frac{\widetilde{\nu}}{\nu_1}$,
we have that, fixing any $0 < \rho < \rho_*$,
\[
\begin{split}
I(u)\ &\ge\ \frac{1}{p} \left(1 - \frac{\widetilde{\nu}}{\nu_1}\right) \int_\Omega A_0(x)\ |\nabla u|^p\ dx
+ \frac{1}{p}\ \left(1 - \frac{\widetilde{\nu}}{\nu_1}\right) \int_\Omega A(x)\ |u|^{ps}\ |\nabla u|^p\ dx\\
&=\ \left(1 - \frac{\widetilde{\nu}}{\nu_1}\right) \rho\ > 0
\end{split}
\]
for all $u \in \bdry{U_\rho}$, i.e., so that $L(u) = \rho$.
\end{proof}

Now, we are ready to prove Theorem \ref{Theorem 1.2}.

\begin{proof}[Proof of Theorem \ref{Theorem 1.2}]
Since $\nu < \nu_1$, taking
\[
F(x,t) = \int_0^t \big[\mu |\tau|^{p(s+1)-2} \tau + g(x,\tau)\big] d\tau,
\]
from \ref{h2} and \ref{h3} it follows that both
\eqref{3.7} and \eqref{3.8} hold;
hence, Lemma \ref{Lemma 3.2} applies to $I(u) = \J(u)$ and 
a radius $\rho_* > 0$ exists so that 
\begin{equation} \label{mp1}
\inf_{u \in \bdry{U_\rho}} \J(u) > 0\qquad \hbox{for all $0 < \rho < \rho_*$.}
\end{equation}
On the other hand, as $\mu > \mu_1$, fixing 
$\widetilde{\mu} \in\ ]\mu_1,\mu[$ 
from \eqref{1.3} it follows that $\varphi \in X \setminus \set{0}$ exists
such that $S(\varphi) < \widetilde{\mu}$, i.e.,
\begin{equation} \label{mp2}
\frac{1}{s + 1} \int_\Omega |\varphi|^{p (s+1)}\ dx 
\ >\ \frac{1}{\widetilde{\mu}} \int_\Omega A(x)\ |\varphi|^{ps}\ |\nabla \varphi|^p\ dx.
\end{equation}
Thus, \eqref{mp2} together with \eqref{1.4} gives
\[
\begin{split}
\J(\sigma \varphi)\ \le\ &- \left(\frac{\mu}{\widetilde{\mu}} - 1\right)\
\frac{\sigma^{p (s+1)}}{p} \ \int_\Omega A(x)\ |\varphi|^{ps}\ |\nabla \varphi|^p\ dx\\ 
&+\ \frac{\sigma^p}{p} \ \int_\Omega A_0(x)\ |\nabla \varphi|^p\ dx 
+ a_3\ |\Omega| + a_4\ \sigma^r\ \int_\Omega |\varphi|^r\ dx
\end{split}
\]
with $\frac{\mu}{\widetilde{\mu}} - 1 > 0$ and $r < p(s+1)$,
so 
\[
\J(\sigma \varphi) \ \to\ -\infty \quad \text{as } \sigma \to + \infty.
\]
Hence, $\sigma_\varphi > 0$ exists so that
\begin{equation} \label{mp3}
\J(\sigma \varphi) < 0\qquad \hbox{for all $\sigma > \sigma_\varphi$.} 
\end{equation}
Now, take $0 < \rho < \rho_*$ as in \eqref{mp1}
and, from \eqref{mp0} with $u_0=\varphi$
and \eqref{mp3}, fix $\sigma_1 > 0$ so large that 
both $\sigma_1 > \sigma_*(\varphi,\rho)$ and $\sigma_1 > \sigma_\varphi$
hold, hence if $u_1\ =\ \sigma_1 \varphi$
we have that
\begin{equation} \label{mp4}
u_1 \in X \setminus \closure{U}_\rho\quad 
\hbox{and}\quad \J(u_1) < 0.
\end{equation}
Summing up, from \eqref{mp1} and \eqref{mp4}
we can take $\beta$ as in \eqref{2.2} with $J(u) = \J(u)$,
moreover, since $\J$ satisfies the \wCPS{\beta} condition 
by Lemma \ref{Lemma 3.1} \ref{Lemma 3.1.ii}, 
Theorem \ref{Theorem 2.4} applies and
$K_\J^\beta \ne \emptyset$. \\
At last, again from \eqref{2.2} and \eqref{mp1} we have that 
\[
\beta \ge \inf_{u \in \bdry{U_\rho}}\ \J(u) > 0 = \J(0),
\]
which gives $0 \notin K_\J^\beta$.
\end{proof}

\begin{remark}\label{trivial}
Consider $g(x,t)$ as in Example \ref{ex1} with
$b \in L^\infty(\Omega)$ and $p < r < p (s + 1)$.
If both $\nu \le \nu_1$ and $\mu \le \mu_1$ hold,
and coefficient $b(x)$ is nonpositive, eventually $b(x) \not\equiv 0$
if $\nu = \nu_1$ and $\mu = \mu_1$, then from \eqref{1.3} it follows that
\[
\langle d\J(u),u\rangle\ \ge\ (\nu_1 - \nu) \int_\Omega |u|^{p}\ dx
+ (\mu_1 - \mu) \int_\Omega |u|^{p (s+1)}\ dx - \int_\Omega b(x)\ |u|^{r}\ dx\
>\ 0
\]
for all $u \in X\setminus \{0\}$.\\ 
Hence, the only critical point of $\J$ is zero, i.e.,
problem \eqref{1.2} with the nonlinear term as in Example \ref{ex1}
admits only the trivial solution.
\end{remark}


\section*{Acknowledgements}
Kanishka Perera would like to thank
the {\sl Dipartimento di Matematica} at {\sl Universit\`a degli Studi di Bari Aldo Moro}
for its kind hospitality, 
as part of this work was carried out when he was visiting Bari. Vice versa,
Anna Maria Candela would like to thank
the {\sl Department of Mathematics and Systems Engineering} at 
{\sl Florida Institute of Technology}
for its kind hospitality, as this work was finished 
when she was visiting Melbourne.\\
The research that led to the present paper was partially supported 
by MUR PRIN 2022 PNRR Research Project 
P2022YFAJH {\sl ``Linear and Nonlinear PDE's: New Directions and Applications''}
and by INdAM-GNAMPA program {\sl ``GNAMPA Professori Visitatori 2022''}.


\def\cdprime{$''$}

\end{document}